\documentclass[11pt]{article}
\usepackage{amsmath}
\usepackage{amsthm}
\usepackage{amssymb}
\usepackage{amscd}
\usepackage{relsize}
\usepackage{enumerate}
\newtheorem{dfn}{Definition}
\newtheorem{teo}{Theorem}

\newtheorem{cor}[teo]{Corollary}

\newcommand{\espan}{\operatorname{span}}
\DeclareMathOperator*{\esup}{ess\,sup}
\DeclareMathOperator*{\einf}{ess\,inf}
\title{{\bf Average sampling in certain subspaces of Hilbert-Schmidt operators on $L^2(\mathbb{R}^d)$}}
\author{
{\bf Antonio G. Garc\'{\i}a}\thanks{E-mail:\texttt{agarcia@math.uc3m.es}}
}
\date{}
\begin{document}
\maketitle
\begin{itemize}
\item[] Departamento de Matem\'aticas, Universidad Carlos III de Madrid, Spain.
\end{itemize}
%%%%%%%%%%%%%%%%%%%%
\begin{abstract}
The concept of translation of an operator allows to consider the  analogous of  shift-invariant subspaces in the class of Hilbert-Schmidt operators. Thus, 
we extend the concept of average sampling to this new setting, and we obtain the corresponding sampling formulas. The key point here is the use of the Weyl transform, a unitary mapping between the space of square integrable functions in the phase space $\mathbb{R}^d\times \widehat{\mathbb{R}}^d$ and the Hilbert space of Hilbert-Schmidt operators on $L^2(\mathbb{R}^d)$, which permits to take advantage of  some well established sampling results.   
\end{abstract}
%%%%%%%%%%%%%%%%%%%%%%%%%%%%%%%%%%%%%%%%%%%%%%%%%%%%%%%%%%%%%%%%%%%%%%

\noindent{\bf AMS}: 42C15; 43A32; 47B10; 94A20.
%%%%%%%%%%%%%%%%%%
\section{Statement of the problem}
\label{section1}
%%%%%%%%%%%%%%%%%%
In this paper a generalized average sampling theory is established for a shift-invariant-like subspace of the class $\mathcal{H}\mathcal{S}(\mathbb{R}^d)$ of Hilbert-Schmidt operators on $L^2(\mathbb{R}^d)$ obtained by translations in a lattice of a fixed  {\em Hilbert-Schmidt operator} $S$. To be more precise, by using conjugation with the {\em time-frequency shift} $\pi(z)$, where $z=(x, \omega)$ belongs to the {\em phase space} $\mathbb{R}^d\times \widehat{\mathbb{R}}^d$, one defines the translation of $S$ by $z\in \mathbb{R}^d\times \widehat{\mathbb{R}}^d$ by 
$\alpha_z(S)=\pi (z) S \pi(z)^*$; remind that $\pi(z)f(t)={\rm e}^{2\pi i \omega \cdot t} f(t-x)$  for $f\in L^2(\mathbb{R}^d)$. If we take a {\em full rank  lattice} $\Lambda$ in $\mathbb{R}^d\times \widehat{\mathbb{R}}^d$ such that the sequence $\{\alpha_\lambda (S)\}_{\lambda \in \Lambda}$ is a {\em Riesz sequence} in $\mathcal{H}\mathcal{S}(\mathbb{R}^d)$ we consider the subspace, analogous of a shift-invariant subspace in $L^2(\mathbb{R}^d)$, given by
\[
V_S^2:=\big\{\sum_{\lambda \in \Lambda} c(\lambda)\, \alpha_\lambda(S)\,\,: \,\, \{c(\lambda)\}_{\lambda \in \Lambda}\in \ell^2(\Lambda)\big\}\,.
\]
Defining the {\em average samples} of any $T\in V_S^2$ at the lattice $\Lambda$ by 
\[
\big\langle T, \alpha_\lambda(Q)\big \rangle_{_{\mathcal{H}\mathcal{S}}}\,,\quad \lambda \in \Lambda\,,
\]
where $Q$ is a fixed element in $\mathcal{H}\mathcal{S}(\mathbb{R}^d)$, not necessarily in $V_S^2$, the aim is to obtain a sampling formula in $V_S^2$ having the form
\[
T=\sum_{\lambda \in \Lambda}\langle T, \alpha_\lambda(Q) \rangle_{_{\mathcal{H}\mathcal{S}}}\, \alpha_\lambda(H) \quad \text{in $\mathcal{H}\mathcal{S}$-norm}\,,
\]
for each $T\in V_S^2$, where the operator $H$ belongs to $V_S^2$ and satisfies that the sequence $\{\alpha_\lambda (H)\}_{\lambda \in \Lambda}$ is a {\em Riesz basis} for $V_S^2$.

\medskip

Thus, we are dealing with a generalization of the usual average sampling in a shift-invariant subspace $V_\varphi^2=\big\{\sum_{\alpha \in \mathbb{Z}^d} c_\alpha\, \varphi(t-\alpha)\,\,:\,\, \{c_\alpha\}\in \ell^2(\mathbb{Z}^d) \big\}$ of $L^2(\mathbb{R}^d)$ generated by the function $\varphi \in L^2(\mathbb{R}^d)$. In this case, for any $f\in V_\varphi^2$ the average samples are defined by 
$\langle f, \psi(\cdot-\alpha)\rangle$, $\alpha \in \mathbb{Z}^d$, where $\psi$ is an {\em average function} in $L^2(\mathbb{R}^d)$, not necessarily in $V_\varphi^2$ (see, for instance, Refs.~\cite{aldroubi:05,garcia:06}). This case was generalized in Ref.~\cite{hector:14} by considering another unitary representation $\{U(t)\}_{t\in \mathbb{R}}$ of $\mathbb{R}$ on $L^2(\mathbb{R})$ instead of the classical one given by the translations $\{T_t\}_{t\in \mathbb{R}}$. In fact, due to the properties of the translation operator $\alpha_z$,  one can consider that $\big\{\alpha_z\big\}_{z\in \mathbb{R}^d\times \widehat{\mathbb{R}}^d}$ is a {\em unitary representation} of the group $\mathbb{R}^d\times \widehat{\mathbb{R}}^d$ on the Hilbert space $\mathcal{H}\mathcal{S}(\mathbb{R}^d)$ of Hilbert-Schmidt operators on $L^2(\mathbb{R}^d)$ (see, for instance, Ref.~\cite{werner:84}).

\medskip  

As in the classical case (see, for instance, Refs.~\cite{aldroubi:05,garcia:08}), the average sampling theory is enriched by considering the multiple generators setting. Here, we consider the subspace of $\mathcal{H}\mathcal{S}(\mathbb{R}^d)$
\[
V_{\bf S}^2=\Big\{\sum_{n=1}^N\sum_{\lambda \in \Lambda} c_n(\lambda)\, \alpha_\lambda(S_n)\,\, :\,\, \{c_n(\lambda)\}_{\lambda \in \Lambda}\in \ell^2(\Lambda)\,,\, n=1, 2, \dots, N \Big\}\,, 
\]
where ${\bf S}=\{S_1, S_2, \dots, S_N\} \subset \mathcal{H}\mathcal{S}(\mathbb{R}^d)$ is the fixed set of generators. The multiple generators case allows to introduce a suitable {\em oversampling} by considering a set of samples $\big\langle T, \alpha_\lambda(Q_m) \big\rangle_{_{\mathcal{H}\mathcal{S}}}$, $\lambda \in \Lambda$, defined from $M\geq N$ fixed operators $Q_m$, $m=1,2 \dots,M$, not necessarily in $V_{\bf S}^2$. In this case the aim is to obtain a sampling formula  having the form
\[
T=\sum_{m=1}^M\sum_{\lambda \in \Lambda}\langle T, \alpha_\lambda(Q_m) \rangle_{_{\mathcal{H}\mathcal{S}}}\, \alpha_\lambda(H_m) \quad \text{in $\mathcal{H}\mathcal{S}$-norm}\,,
\]
for each $T\in V_{\bf S}^2$, where the operators $H_m$, $m=1,2,\dots,M$, belong to $V_{\bf S}^2$, and satisfy that the sequence $\{\alpha_\lambda (H_m)\}_{\lambda \in \Lambda;\, m=1,2 \dots,M}$ is a {\em frame} for $V_{\bf S}^2$.

\medskip 

On the other hand, the average samples of any $T=\sum_{n=1}^N\sum_{\lambda \in \Lambda} c_n(\lambda)\, \alpha_\lambda(S_n)$ in $V_{\bf S}^2$ 
can be expressed as a {\em discrete convolution system} in the product Hilbert space $\ell^2_{_N}(\Lambda):=\ell^2(\Lambda)\times \dots \times \ell^2(\Lambda)$ ($N$ times), i.e.,
\[
\big\langle T, \alpha_\lambda(Q_m) \big\rangle_{_{\mathcal{H}\mathcal{S}}}=\sum_{n=1}^N \big(a_{m,n} \ast_\Lambda c_n\big)(\lambda)\,,\quad \lambda \in \Lambda\,,\,\, \text{$n=1, 2,\dots, N$\, and \,$m=1,2 \dots,M$}\,,
\]
for some $a_{m,n} \in \ell^2(\Lambda)$ (see the details in Section \ref{section4}). Thus, we borrow the sampling scheme followed in  Ref.~\cite{garcia:19} in order  to generalize the above average sampling to the subspace $V_{\bf S}^2$ of $\mathcal{H}\mathcal{S}(\mathbb{R}^d)$. In that reference, discrete convolution systems on discrete abelian groups are proposed as a unifying strategy in sampling theory. This is done once we have taken into account that the {\em Weyl transform} is a unitary operator from $L^2(\mathbb{R}^d\times \widehat{\mathbb{R}}^d)$ onto $\mathcal{H}\mathcal{S}(\mathbb{R}^d)$ allowing to transfer sampling in a shift-invariant subspace of $L^2(\mathbb{R}^d\times \widehat{\mathbb{R}}^d)$ into sampling in the subspace $V_{\bf S}^2$ of $\mathcal{H}\mathcal{S}(\mathbb{R}^d)$. Thus, the Weyl transform will be the cornerstone of this work along with the well-known average sampling in shift-invariant subspaces. It is worth to mention that, instead of the Weyl transform, we could use the Kohn-Niremberg transform.

\medskip

The paper is organized as follows: Section \ref{section2} introduces the novel preliminaries needed in the sequel; these  comprise Hilbert-Schmidt operators and the Weyl transform, the concept of translation of an operator, {\em symplectic Fourier series} and Riesz sequences of translation operators in 
$\mathcal{H}\mathcal{S}(\mathbb{R}^d)$. For the theory of bases and frames in a Hilbert space we cite Ref.~\cite{ole:16}. In Section \ref{section3}, a generalized average sampling theorem is obtained for the one generator case $V_S^2$. In Section \ref{section4} the former sampling theory is developed for the multiple generators case $V_{\bf S}^2$. As it was said before, this study relies on the theory of bounded discrete convolution systems $\ell^2_{_N}(\Lambda) \rightarrow \ell^2_{_M}(\Lambda)$ and their relationship with  
frames of translates in $\ell^2_{_N}(\Lambda)$. The needed results on these topics will be also briefly reminded in this section. 
%%%%%%%%%%%%%%
\section{Preliminaries}
\label{section2}
%%%%%%%%%%%%%%%%%%%%%%%%%%%%%%%%
For the sake of completeness, in this section we briefly introduce the novel mathematical tools used throughout the work. For the needed theory of bases and frames in a Hilbert space we merely make reference to \cite{ole:16}; it mainly comprises Riesz sequences, dual Riesz bases and frames and its duals in a separable Hilbert space.
%%%%%%%%%%%%%%%%%%%%%%%%%%%%%%%%
\subsection{Hilbert-Schmidt operators and the Weyl transform}
%%%%%%%%%%%%%%%%%%%%%%%%%%%%%%%%
\subsubsection*{Hilbert-Schmidt operators on $L^2(\mathbb{R}^d)$}
%%%%%%%%%%%%%%%%%%%%%%%%%%%%%%%%
There are different ways to introduce the class of Hilbert-Schmidt operators in a Hilbert space, $L^2(\mathbb{R}^d)$ in our case. We follow that using the {\em Schmidt decomposition} (singular value decomposition) of a compact operator on $L^2(\mathbb{R}^d)$ (see, for instance, Ref.~\cite{conway:00}). Namely,  for a compact operator $S$ on $L^2(\mathbb{R}^d)$ there exist two orthonormal sequences $\{x_n\}_{n\in \mathbb{N}}$ and $\{y_n\}_{n\in \mathbb{N}}$ in $L^2(\mathbb{R}^d)$ and a bounded sequence of positive numbers $\{s_n(S)\}_{n\in \mathbb{N}}$ ({\em singular values} of $S$) such that 
\[
S=\sum_{n\in \mathbb{N}} s_n(S)\, x_n \otimes y_n\,,
\]
with convergence of the series in the operator norm. Here,  $x_n \otimes y_n$ denotes the rank one operator $ \big(x_n \otimes y_n\big)(e)=\big\langle e, y_n\big\rangle_{L^2} x_n$ for $e\in L^2(\mathbb{R}^d)$. For $1\le p<\infty$  we define the {\em Schatten-$p$ class} $\mathcal{T}^p$ class by
\[
\mathcal{T}^p:=\big\{ S \,\text{ compact on $L^2(\mathbb{R}^d)$}\, \,:\,\, \{s_n(S)\}_{n\in \mathbb{N}} \in \ell^p(\mathbb{N})\big\}\,.
\]
The Schatten-$p$ class $\mathcal{T}^p$ is a Banach space with the norm $\|S\|_{\mathcal{T}^p}^p=\sum_{n\in \mathbb{N}} s_n(S)^p$.

\medskip 

In particular, for $p=1$ we obtain the so-called {\em trace class operators $\mathcal{T}^1$}. The {\em trace} defined by ${\rm tr}(S)=\sum_{n\in \mathbb{N}} \langle Se_n, e_n\rangle_{L^2}$ is a well-defined bounded linear functional on $\mathcal{T}^1$, and independent of the used orthonormal basis $\{e_n\}_{n\in \mathbb{N}}$ in $L^2(\mathbb{R}^d)$. 

\medskip

For $p=2$ we obtain the class of {\em Hilbert-Schmidt operators} $\mathcal{H}\mathcal{S}(\mathbb{R}^d):=\mathcal{T}^2$. The space $\mathcal{H}\mathcal{S}(\mathbb{R}^d)$ endowed of the inner product $\big\langle S, T \big\rangle_{\mathcal{H}\mathcal{S}}={\rm tr}(ST^*)$ is a Hilbert space. For the norm of $S\in \mathcal{H}\mathcal{S}(\mathbb{R}^d)$ we have
\[
\|S\|_{\mathcal{H}\mathcal{S}}^2={\rm tr}(SS^*)=\sum_{n\in \mathbb{N}}\|S^*(e_n)\|^2_{L^2}=\sum_{n\in \mathbb{N}}\|S(e_n)\|^2_{L^2}=\sum_{n\in \mathbb{N}} s_n^2(S)\,.
\]
A Hilbert-Schmidt operator $S\in \mathcal{H}\mathcal{S}(\mathbb{R}^d)$ can be seen also as a compact operator on $L^2(\mathbb{R}^d)$ defined for each $f\in L^2(\mathbb{R}^d)$ by
\[
Sf(t)=\int_{\mathbb{R}^d} k_{_S}(t,x) f(x) dx\quad \text{ a.e. $t\in \mathbb{R}^d$}\,,
\]
with kernel $k_{_S} \in L^2(\mathbb{R}^{2d})$. Besides, $\big\langle S, T \big\rangle_{\mathcal{H}\mathcal{S}}=\big\langle k_{_S}, k_{_T}\big\rangle_{L^2(\mathbb{R}^{2d})}$\, for $S, T\in \mathcal{H}\mathcal{S}(\mathbb{R}^d)$.
%%%%%%%%%%%%%%%%%%%
\subsubsection*{The Weyl transform}
%%%%%%%%%%%%%%%%%%%
We introduce the Weyl transform in $L^2(\mathbb{R}^d\times \widehat{\mathbb{R}}^d)$,  the setting used in our context. More information and details about this transform, also valid in more general settings, can be found in Refs.~\cite{grochenig:01,pool:66,skret:20,werner:84}.

\medskip

The Weyl transform $L^2(\mathbb{R}^d\times \widehat{\mathbb{R}}^d) \ni f\longmapsto L_f \in \mathcal{H}\mathcal{S}(\mathbb{R}^d)$ is a unitary operator where 
$L_f: L^2(\mathbb{R}^d) \rightarrow L^2(\mathbb{R}^d)$ is the Hilbert-Schmidt operator defined in weak sense by
\[
\big\langle L_f \phi, \psi \big\rangle_{L^2(\mathbb{R}^d)}=\big\langle f, W(\psi, \phi)\big\rangle_{L^2(\mathbb{R}^d\times \widehat{\mathbb{R}}^d)}\,, \quad \phi, \psi \in L^2(\mathbb{R}^d)\,,
\]
where
\[
W(\psi, \phi)(x,\omega )=\int_{\mathbb{R}^d}\psi\big(x+\frac{t}{2}\big)\, \overline{\phi\big(x-\frac{t}{2}\big)}\,{\rm e}^{-2\pi i \,{\omega \cdot t}}dt\,,\quad (x,\omega ) \in \mathbb{R}^d\times \widehat{\mathbb{R}}^d\,,
\]
is the {\em cross-Wigner distribution} of the functions $\psi, \phi \in L^2(\mathbb{R}^d)$ (see Ref.~\cite{folland:89}).

\medskip

Thus, for each operator $S\in \mathcal{H}\mathcal{S}(\mathbb{R}^d)$ there exists a unique function $a_{_S}\in L^2(\mathbb{R}^d\times \widehat{\mathbb{R}}^d)$, called its {\em Weyl symbol}, and  such that 
\[
\langle S, T\rangle_{\mathcal{H}\mathcal{S}}=\langle a_{_S}, a_{_T} \rangle_{L^2(\mathbb{R}^d\times \widehat{\mathbb{R}}^d)}\quad \text{for each $S, T \in \mathcal{H}\mathcal{S}(\mathbb{R}^d)$}\,.
\]
In Ref.~\cite{pfander:13} a sampling theory for operators with bandlimited Kohn-Niremberg symbol is developed; see also Refs.~\cite{grochenig:14,krahmer:14,pfander:16}. The {\em Kohn-Niremberg symbol} is an alternative to the Weyl symbol which relates  Hilbert-Schmidt operators with pseudo-differential calculus. See Ref.~\cite{grochenig:01} for the transition between the Weyl calculus and the Kohn-Niremberg calculus.
%%%%%%%%%%%%%%%%%%%%%
\subsection{Translation of operators}
%%%%%%%%%%%%%%%%%%%%%
For $z=(x, \omega) \in \mathbb{R}^d\times \widehat{\mathbb{R}}^d$, the {\em time-frequency shift} operator $\pi(z): L^2(\mathbb{R}^d) \rightarrow L^2(\mathbb{R}^d)$ is defined as 
\[
\pi(z) \varphi(t)={\rm e}^{2\pi i \omega\cdot t}\varphi(t-x)\quad \text{for $\varphi \in L^2(\mathbb{R}^d)$}\,.
\]
It is used to define the {\em short-time Fourier transform} (Gabor transform) $V_\psi\varphi$ of $\varphi$ with respect to the window $\psi$, both in $L^2(\mathbb{R}^d)$, by
\[
V_\psi\varphi(z)=\big\langle \varphi, \pi(z)\psi \big\rangle_{L^2}\,,\quad z \in \mathbb{R}^d\times \widehat{\mathbb{R}}^d\,.
\]
The adjoint operator of $\pi(z)$ is $\pi(z)^*={\rm e}^{-2\pi i x\cdot \omega}\, \pi(-z)$ for $z=(x, \omega)$.
By using conjugation with $\pi(z)$ one can define the translation by $z\in \mathbb{R}^d\times \widehat{\mathbb{R}}^d$ of an operator $S \in \mathcal{H}\mathcal{S}(\mathbb{R}^d)$. Namely,
\[
\alpha_z(S):=\pi(z)\,S\,\pi(z)^*\,,\quad z \in \mathbb{R}^d\times \widehat{\mathbb{R}}^d\,.
\]
As an example, for $\varphi, \psi \in L^2(\mathbb{R}^d)$ we have $\alpha_z(\varphi \otimes \psi)=[\pi(z)\varphi] \otimes [\pi(z)\psi]$, \,$z \in \mathbb{R}^d\times \widehat{\mathbb{R}}^d$.
Since translation $\alpha_z$ defines a unitary operator on $\mathcal{H}\mathcal{S}(\mathbb{R}^d)$ and $\alpha_z \alpha_{z'}=\alpha_{z+z'}$ for $z, z' \in \mathbb{R}^d\times \widehat{\mathbb{R}}^d$ we can consider  
$\big\{\alpha_z\big\}_{z\in \mathbb{R}^d\times \widehat{\mathbb{R}}^d}$ as a {\em unitary representation} of the group $\mathbb{R}^d\times \widehat{\mathbb{R}}^d$ on the Hilbert space 
$\mathcal{H}\mathcal{S}(\mathbb{R}^d)$. For more properties and applications see, for instance, Refs.~\cite{luef:18,skret:20, werner:84}.

%%%%%%%%%%%%%%%%%%%%%
\subsection{Symplectic Fourier series}
%%%%%%%%%%%%%%%%%%%%%
Let $\Lambda$ be a {\em full rank lattice} in $\mathbb{R}^d\times \widehat{\mathbb{R}}^d$, i.e., $\Lambda=A\mathbb{Z}^{2d}$ with $A\in GL(2d,\mathbb{R})$ and volume $|\Lambda|=\det A$. Its dual group $\widehat{\Lambda}$ is identified with $(\mathbb{R}^d\times \widehat{\mathbb{R}}^d)/\Lambda^\circ$, where $\Lambda^\circ$ is the {\em annihilator group}
\[
\Lambda^\circ=\big\{\lambda^\circ \in \mathbb{R}^d\times \widehat{\mathbb{R}}^d\, \, :\, \, {\rm e}^{2\pi i\, \sigma(\lambda^\circ , \lambda)}=1 \text{ for all $\lambda \in \Lambda$} \big\}\,,
\]
where $\sigma$ denotes the {\em standard symplectic form} $\sigma (z,z')=\omega\cdot x'-\omega'\cdot x$ \, for $z=(x, \omega)$ and $z'=(x',\omega')$ in $\mathbb{R}^d\times \widehat{\mathbb{R}}^d$. Notice that the dual group $\widehat{\Lambda}$ is compact. The group $\Lambda^\circ$ is itself a lattice: the so-called {\em adjoint lattice} of $\Lambda$. The {\em symplectic characters} $\chi_z(z'):={\rm e}^{2\pi i \,\sigma(z,z')}$ are the natural way of identifying the group 
$\mathbb{R}^d\times \widehat{\mathbb{R}}^d$ with its dual group via the bijection $z\mapsto \chi_z$. Sometimes we will identify the phase space $\mathbb{R}^d\times \widehat{\mathbb{R}}^d$ with its isomorphic space $\mathbb{R}^{2d}$.

\medskip

The  Fourier transform of $c\in \ell^1(\Lambda)$ is the {\em symplectic Fourier series}
\[
\mathcal{F}_\sigma^\Lambda(c)(\dot z):=\sum_{\lambda \in \Lambda} c(\lambda)\, {\rm e}^{2\pi i \sigma(\lambda, z)}\,, \quad \dot z \in (\mathbb{R}^d\times \widehat{\mathbb{R}}^d)/\Lambda^\circ\,,
\]
where $\dot z$ denotes the image of $z$ under the natural quotient map $\mathbb{R}^d\times \widehat{\mathbb{R}}^d \rightarrow (\mathbb{R}^d\times \widehat{\mathbb{R}}^d)/\Lambda^\circ$.

\medskip

Since $\mathcal{F}_\sigma^\Lambda$ is a Fourier transform it extends to a unitary mapping $\mathcal{F}_\sigma^\Lambda : \ell^2(\Lambda) \rightarrow L^2(\,\widehat{\Lambda}\,)$. It satisfies $\mathcal{F}_\sigma^\Lambda(c\ast_\Lambda d)=\mathcal{F}_\sigma^\Lambda(c)\,\mathcal{F}_\sigma^\Lambda(d)$ for $c\in \ell^1(\Lambda)$ and $d\in \ell^2(\Lambda)$. Moreover, if $c, d \in \ell^2(\Lambda)$ with $c\ast_\Lambda d \in \ell^2(\Lambda)$, then $\mathcal{F}_\sigma^\Lambda(c\ast_\Lambda d)=\mathcal{F}_\sigma^\Lambda(c)\,\mathcal{F}_\sigma^\Lambda(d)$. As usual, the convolution $\ast_\Lambda$ of two sequences $c, d$  is defined by
\[
\big(c \ast_\Lambda d\big)(\lambda)=\sum_{\lambda'\in \Lambda} c(\lambda')\, d(\lambda-\lambda'),\quad \lambda \in \Lambda\,.
\]
For more details, see, for instance, Refs.~\cite{deitmar:14,folland:95, fuhr:05,skret:20}.
%%%%%%%%%%%%%%%%%%%%%%%%%%%%%%%%%%%%%
\subsection{Riesz sequences of translation operators in $\mathcal{H}\mathcal{S}(\mathbb{R}^d)$}
%%%%%%%%%%%%%%%%%%%%%%%%%%%%%%%%%%%%%
The Weyl transform $f\mapsto L_f$ is a unitary operator $L^2(\mathbb{R}^d\times \widehat{\mathbb{R}}^d) \rightarrow \mathcal{H}\mathcal{S}(\mathbb{R}^d)$ which respects translations in the sense that
\[
L_{T_zf}=\alpha_z(L_f)\quad \text{for $f\in L^2(\mathbb{R}^d\times \widehat{\mathbb{R}}^d)$ and $z=(x,\omega) \in \mathbb{R}^d\times \widehat{\mathbb{R}}^d$}\,.
\]
These two properties are crucial throughtout this work. In particular, as it was pointed out in Ref.~\cite{skret:20}, for fixed $S\in \mathcal{H}\mathcal{S}(\mathbb{R}^d)$ with Weyl symbol $a_{_S} \in L^2(\mathbb{R}^d\times \widehat{\mathbb{R}}^d)$ and lattice $\Lambda$ in $\mathbb{R}^d\times \widehat{\mathbb{R}}^d$, the sequence 
$\{\alpha_\lambda (S)\}_{\lambda \in \Lambda}$ is a Riesz sequence in $\mathcal{H}\mathcal{S}(\mathbb{R}^d)$, i.e., a Riesz basis for $V_S^2:=\overline{\espan}_{\mathcal{H}\mathcal{S}}\big\{ \alpha_\lambda(S)\big\}_{\lambda \in \Lambda}$, if and only if the sequence $\{T_\lambda (a_{_S})\}_{\lambda \in \Lambda}$ is a Riesz sequence in $L^2(\mathbb{R}^d\times \widehat{\mathbb{R}}^d)$, i.e., a Riesz basis for the shift-invariant subspace $V_{a_S}^2$ in $L^2(\mathbb{R}^d\times \widehat{\mathbb{R}}^d)$ generated by $a_{_S}$. 

A necessary and sufficient condition to be $\{\alpha_\lambda (S)\}_{\lambda \in \Lambda}$ a Riesz sequence in $\mathcal{H}\mathcal{S}(\mathbb{R}^d)$
is given in Ref.~\cite[Theorem 6.1]{skret:20}. There, it is assumed that $S\in \mathcal{B}$, a Banach space of continuous operators with Weyl symbol $a_{_S}$ in the {\em Feichtinger's algebra} $\mathcal{S}_0(\mathbb{R}^d\times \widehat{\mathbb{R}}^d)$ \cite{jakobsen:18}; in essence, 
$\mathcal{B}$ consists of trace class operators on $L^2(\mathbb{R}^d)$ with a norm-continuous inclusion 
$\iota : \mathcal{B} \hookrightarrow \mathcal{T}^1$  (see the details in Refs.~\cite{grochenig:99,skret:20}). The alluded necessary and sufficient  condition is that the continuous function
\[
P_{\Lambda^\circ}(|\mathcal{F}_W(S)|^2(\dot z):=\frac{1}{|\Lambda|}\sum_{\lambda^\circ \in \Lambda^\circ} |\mathcal{F}_W(S)(z+\lambda^\circ)|^2\,,\quad z\in \mathbb{R}^d\times \widehat{\mathbb{R}}^d\,,
\] 
has no zeros in $\widehat{\Lambda}$. It involves the {\em periodization operator} $P_{\Lambda^\circ}$ in $\Lambda^\circ$ and the {\em Fourier-Wigner transform} $\mathcal{F}_W$ of operator $S$. In this case, we have that $\mathcal{F}_W(S)=\mathcal{F}_\sigma(a_{_S})$, where $\mathcal{F}_\sigma$ denotes the {\em symplectic Fourier transform} of $a_{_S}$ defined by
\[
\mathcal{F}_\sigma(a_{_S})(z):=\int_{\mathbb{R}^d \times \widehat{\mathbb{R}}^d} a_{_S}(z')\, {\rm e}^{-2\pi i \,\sigma(z, z')}dz'\,,\quad z\in \mathbb{R}^d \times \widehat{\mathbb{R}}^d\,,
\]
where $\sigma$ denotes the standard symplectic form in $\mathbb{R}^d \times \widehat{\mathbb{R}}^d$. The Wigner-Fourier transform of an operator $S$ is defined as the function
\[
\mathcal{F}_W(S)(z):={\rm e}^{-\pi i \,x\cdot \omega}\, {\rm tr}[\pi(-z)S]\,, \quad z=(x,\omega)\in \mathbb{R}^d\times \widehat{\mathbb{R}}^d\,.
\]
See the details in Ref.~\cite{skret:20}. A similar result for a rank one operator $S=\psi\otimes \phi$, where $\psi, \phi \in  L^2(\mathbb{R}^d)$, can be found in Ref.~\cite{benedetto:06}.

\medskip

Analogously, a necessary and sufficient condition can be found for the multiply generated  case. Indeed, let ${\bf S}=\{S_1, S_2, \dots, S_N\}$ be a fixed subset of $\mathcal{H}\mathcal{S}(\mathbb{R}^d)$ and let $\Lambda$ be a lattice in $\mathbb{R}^d\times \widehat{\mathbb{R}}^d$; we look for a necessary and sufficient condition such that  $\{\alpha_\lambda (S_n)\}_{\lambda \in \Lambda;\, n=1,2,\dots,N}$ is a Riesz sequence for $\mathcal{H}\mathcal{S}(\mathbb{R}^d)$, i.e., a Riesz basis for the closed subspace
\[
V_{\bf S}^2:=\overline{\espan}_{\mathcal{H}\mathcal{S}}\big\{ \alpha_\lambda(S_n)\big\}_{\lambda \in \Lambda;\,n=1,2, \dots,N}\subset \mathcal{H}\mathcal{S}(\mathbb{R}^d)\,.
\]
As indicated above, it will be a Riesz sequence in $\mathcal{H}\mathcal{S}(\mathbb{R}^d)$ if and only if the sequence $\{T_\lambda (a_{_{S_n}})\}_{\lambda \in \Lambda;\,n=1,2,\dots,N}$ is a Riesz sequence in $L^2(\mathbb{R}^d\times \widehat{\mathbb{R}}^d)$. To this end,  we introduce the $N\times N$ matrix-valued function
\[
G_{\bf S}^\sigma(z):=\sum_{\lambda^\circ \in \Lambda^\circ} \mathcal{F}_\sigma(a_{_{\bf S}})(z+\lambda^\circ)\, \overline{ \mathcal{F}_\sigma(a_{_{\bf S}})(z+\lambda^\circ)}^\top\,, \quad z\in \mathbb{R}^d\times \widehat{\mathbb{R}}^d\,,
\]
where $\mathcal{F}_\sigma(a_{_{\bf S}})=\big(\mathcal{F}_\sigma(a_{_{S_1}}), \mathcal{F}_\sigma(a_{_{S_2}}), \dots,\mathcal{F}_\sigma(a_{_{S_N}})\big)^\top$. It is known (see, for instance, Ref.~\cite{aldroubi:05}) that the sequence $\{T_\lambda (a_{_{S_n}})\}_{\lambda \in \Lambda;\,n=1,2,\dots,N}$ is a Riesz sequence in $L^2(\mathbb{R}^d\times \widehat{\mathbb{R}}^d)$ if and only if there exist two constants $0<m\le M$ such that $m\,\mathbb{I}_N \le G_{\bf S}^\sigma(z) \le M\,\mathbb{I}_N$, a.e. $z\in \mathbb{R}^d\times \widehat{\mathbb{R}}^d$, where $\mathbb{I}_N$ denotes the $N\times N$ identity matrix. 

\medskip

Assuming as before that $S_n\in \mathcal{B}$, $n=1,2,\dots,N$, the functions $\mathcal{F}_\sigma(a_{_{S_n}})$ are continuous and $\mathcal{F}_W(S_n)=\mathcal{F}_\sigma(a_{_{S_n}})$ for $n=1,2,\dots,N$. Hence,  the above necessary and sufficient condition can be expressed in terms of the Wigner-Fourier transforms of operators $S_n$ by defining the $N\times N$ matrix-valued function
\[
G_{\bf S}^W(z):=\sum_{\lambda^\circ \in \Lambda^\circ} \mathcal{F}_W({\bf S})(z+\lambda^\circ)\, \overline{\mathcal{F}_W({\bf S})(z+\lambda^\circ)}^\top\,, \quad z\in \mathbb{R}^d\times \widehat{\mathbb{R}}^d\,,
\]
where $\mathcal{F}_W({\bf S})=\big(\mathcal{F}_W(S_1), \mathcal{F}_W(S_2), \dots,\mathcal{F}_W(S_N)\big)^\top$. The condition reads: 
\[
m\,\mathbb{I}_{_N} \le G_{\bf S}^W(z) \le M\,\mathbb{I}_{_N} \quad  \text{for any $z\in \mathbb{R}^d\times \widehat{\mathbb{R}}^d$}\,.
\]
The above inequalities mean that, for each $z\in \mathbb{R}^d\times \widehat{\mathbb{R}}^d$, the hermitian matrix $G_{\bf S}^W(z)$ satisfies
\[
m\, \|{\bf x}\|^2 \le \big\langle G_{\bf S}^W(z) {\bf x}, {\bf x} \big\rangle_{\mathbb{C}^N} \le M\,\|\bf {x}\|^2 \quad \text{for all ${\bf x}\in \mathbb{C}^N$}\,.
\]
%%%%%%%%%%%%%%%%%%
\section{An average sampling result}
\label{section3}
%%%%%%%%%%%%%%%%%%
Let $S\in \mathcal{H}\mathcal{S}(\mathbb{R}^d)$ be a Hilbert-Schmidt operator with Weyl symbol $a_{_S}\in L^2(\mathbb{R}^d\times \widehat{\mathbb{R}}^d)$, i.e., $L_{a_S}=S$, and let $\Lambda$ be a full rank lattice in $\mathbb{R}^d\times \widehat{\mathbb{R}}^d$. Associated to the operator $S$ we consider the invariant subspace in $\mathcal{H}\mathcal{S}(\mathbb{R}^d)$ defined as $V_S^2:=\overline{\espan}_{\mathcal{H}\mathcal{S}}\big\{ \alpha_\lambda(S)\big\}_{\lambda \in \Lambda}$, where 
$\alpha_\lambda(S)=\pi(\lambda)\,S\, \pi(\lambda)^*$, \,$\lambda \in \Lambda$. Assuming that $\{\alpha_\lambda (S)\}_{\lambda \in \Lambda}$ is a Riesz sequence for $\mathcal{H}\mathcal{S}(\mathbb{R}^d)$, the subspace $V_S^2$  can be expressed as
\[
V_S^2=\big\{\sum_{\lambda \in \Lambda} c(\lambda)\, \alpha_\lambda(S)\,\,: \,\, \{c(\lambda)\}_{\lambda \in \Lambda}\in \ell^2(\Lambda)\big\}\,.
\]
Associated with $V_S^2$ we consider the shift-invariant subspace in $L^2(\mathbb{R}^d\times \widehat{\mathbb{R}}^d)$ generated by $a_S$, i.e.,
\[
V_{a_S}^2=\big\{\sum_{\lambda \in \Lambda} c(\lambda)\, T_\lambda  a_{_S}\,\, :\,\, \{c(\lambda)\}_{\lambda \in \Lambda}\in \ell^2(\Lambda)\big\}\,.
\] 
Since the Weyl transform is a unitary operator between $L^2(\mathbb{R}^{2d})$ and $\mathcal{H}\mathcal{S}(\mathbb{R}^d)$ and $L_{T_zf}=\alpha_z(L_f)$, the sequence $\{\alpha_\lambda (S)\}_{\lambda \in \Lambda}$ is a Riesz sequence for $\mathcal{H}\mathcal{S}(\mathbb{R}^d)$ if and only if $\{T_\lambda  a_S\}_{\lambda \in \Lambda}$ is a Riesz sequence for $L^2(\mathbb{R}^d\times \widehat{\mathbb{R}}^d)$.

\medskip

Our sampling results rely on the following isomorphism $\mathcal{T}_S$ which involves the spaces $\ell^2(\Lambda)$, $V_{a_S}^2$ and $V_{S}^2$. Namely,
\begin{equation}
\label{isos}
\begin{array}[c]{cccccc}
  \mathcal{T}_S:&\ell^2(\Lambda) & \longrightarrow & V_{a_S}^2 \subset L^2(\mathbb{R}^d\times \widehat{\mathbb{R}}^d) & \longrightarrow & V_{S}^2 \subset \mathcal{H}\mathcal{S}(\mathbb{R}^d) 
  \phantom{\dfrac{a}{b}} \\
       & \{c(\lambda)\}_{\lambda \in \Lambda} & \longmapsto & \displaystyle{\sum_{\lambda \in \Lambda} c(\lambda)\, T_\lambda  a_S}  & \longmapsto  & \displaystyle{\sum_{\lambda \in \Lambda} c(\lambda)\, \alpha_\lambda(S)}\,,
\end{array}
\end{equation}
is the composition of the  isomorphism $\mathcal{T}_{a_S}$ between $\ell^2(\Lambda)$ and $V_{a_S}^2$ which maps the standard orthonormal basis 
$\{\delta_\lambda\}_{\lambda \in \Lambda}$ for $\ell^2(\Lambda)$ onto the Riesz basis $\{T_\lambda  a_{_S}\}_{\lambda \in \Lambda}$ for $V_{a_S}^2$, and the Weyl transform between $V_{a_S}^2$ and $V_{S}^2$.

\medskip

Next we define the {\em generalized average samples} for any $T=\sum_{\lambda \in \Lambda} c(\lambda)\, \alpha_\lambda(S)$ in $V_{S}^2$. Namely, for a fixed $Q\in \mathcal{H}\mathcal{S}(\mathbb{R}^d)$, not necessarily in $V_{S}^2$, we define the samples $\{s_{_T}(\lambda)\}_{\lambda \in \Lambda}$ of $T$ at the lattice 
$\Lambda$ by
\begin{equation}
\label{samples1}
s_{_T}(\lambda):=\big\langle T, \alpha_\lambda(Q)\big \rangle_{_{\mathcal{H}\mathcal{S}}}\,, \quad \lambda \in \Lambda\,.
\end{equation}
The first task is to obtain a more suitable expression for these samples. Indeed, for  the sample $s_T(\lambda)$, $\lambda \in \Lambda$, of $T=\sum_{\lambda \in \Lambda} c(\lambda)\, \alpha_\lambda(S)$ we have
\[
\begin{split}
s_{_T}(\lambda):=&\big\langle T, \alpha_\lambda(Q) \big\rangle_{_{\mathcal{H}\mathcal{S}}}=\big\langle a_{_T}, T_\lambda a_{_Q} \big\rangle_{_{L^2(\mathbb{R}^d\times \widehat{\mathbb{R}}^d)}}=
\big\langle \sum_{\lambda' \in \Lambda} c(\lambda')\, T_{\lambda'}  a_{_S}, T_\lambda a_{_Q} \big\rangle_{_{L^2(\mathbb{R}^d\times \widehat{\mathbb{R}}^d)}} \\
 =&\sum_{\lambda' \in \Lambda} c(\lambda')\big \langle T_{\lambda'}  a_{_S}, T_\lambda a_{_Q} \big\rangle_{_{L^2(\mathbb{R}^d\times \widehat{\mathbb{R}}^d)}}=\sum_{\lambda' \in \Lambda} c(\lambda') \big\langle a_{_S}, T_{\lambda-\lambda'} a_{_Q} \big\rangle = \big(\mathbf{c} \ast_\Lambda \mathbf{q}\big)(\lambda)\,,
\end{split}
\]
where $\mathbf{q}=\{q(\lambda)\}_{\lambda \in \Lambda}$ with $q(\lambda):=\big\langle a_{_S}, T_\lambda a_{_Q} \big\rangle_{_{L^2(\mathbb{R}^d\times \widehat{\mathbb{R}}^d)}}$, $\lambda \in \Lambda$, and $\mathbf{c}=\{c(\lambda)\}_{\lambda \in \Lambda}$. Notice that $\mathbf{q} \in \ell^2(\Lambda)$ since, in particular, $\{T_\lambda a_{_S}\}$ is a Bessel sequence in $L^2(\mathbb{R}^d\times \widehat{\mathbb{R}}^d)$.

\medskip

The main aim is the stable recovery  of any $T\in V_{S}^2$ from the data samples $\{s_T(\lambda)\}_{\lambda \in \Lambda}$ given in Eq.~\ref{samples1}. This is equivalent, via the isomorphism $\mathcal{T}_S$, to that the convolution operator  $\mathbf{c} \mapsto \mathbf{c}\ast_\Lambda \mathbf{q}$ defines an isomorphism 
$\ell^2(\Lambda) \rightarrow \ell^2(\Lambda)$. 

\medskip

On the other hand, we have that
\[
s_T(\lambda)=\big\langle T, \alpha_\lambda(Q) \big\rangle_{_{\mathcal{H}\mathcal{S}}}=\big(\mathbf{c} \ast_\Lambda \mathbf{q}\big)(\lambda)=\big\langle \mathbf{c}, T_\lambda \mathbf{q}^*\big\rangle_{\ell^2(\Lambda)}\,,\quad \lambda \in \Lambda\,,
\]
where $T_\lambda$ denotes the translation by $\lambda$ in $\ell^2(\Lambda)$, and $\mathbf{q}^*$ denotes the {\em involution} of $\mathbf{q}$ in $\ell^2(\Lambda)$, i.e., 
$q^*(\lambda)=\overline{q(-\lambda)}$, $\lambda \in \Lambda$, i.e., $T_\lambda q^*(\mu)=q^*(\mu-\lambda)$, $\mu \in  \Lambda$. As a consequence, the convolution operator  $\mathbf{c} \mapsto c\ast_\Lambda \mathbf{q}$ is an isomorphism in 
$\ell^2(\Lambda)$ if and only if the sequence $\{T_\lambda \mathbf{q}^*\}_{\lambda \in \Lambda}$ is a Riesz basis for $\ell^2(\Lambda)$.

\medskip

The convolution operator  $\mathbf{c} \mapsto \mathbf{c}\ast_\Lambda \mathbf{q}$ is a well-defined operator $\ell^2(\Lambda) \rightarrow \ell^2(\Lambda)$ (and consequently bounded) if and only if $\esup_{\xi \in \widehat{\Lambda}} |\mathcal{F}_\sigma^\Lambda(\mathbf{q})(\xi)| <\infty$, where $\mathcal{F}_\sigma^\Lambda(\mathbf{q})$ denotes the symplectic Fourier series of 
$\mathbf{q}$, i.e., the Fourier transform associated to the discrete group $\Lambda$. Besides, the convolution operator  $\mathbf{c} \mapsto \mathbf{c}\ast_\Lambda \mathbf{q}$ is bijective if and only if 
$0<\einf_{\xi \in \widehat{\Lambda}} |\mathcal{F}_\sigma^\Lambda(\mathbf{q})(\xi)|$. 

In summary, the sequence $\{T_\lambda \mathbf{q}^*\}_{\lambda \in \Lambda}$ is a Riesz basis for 
$\ell^2(\Lambda)$ if and only if
\[
0<\einf_{\xi \in \widehat{\Lambda}} |\mathcal{F}_\sigma^\Lambda(\mathbf{q})(\xi)| \le \esup_{\xi \in \widehat{\Lambda}} |\mathcal{F}_\sigma^\Lambda(\mathbf{q})(\xi)| <\infty\,.
\]
In this case, the dual basis of $\{T_\lambda \mathbf{q}^*\}_{\lambda \in \Lambda}$ has the form $\{T_\lambda \mathbf{p}\}_{\lambda \in \Lambda}$, where $\mathbf{p}\in \ell^2(\Lambda)$ satisfies that $\mathcal{F}_\sigma^\Lambda(\mathbf{p})=1/\mathcal{F}_\sigma^\Lambda(\mathbf{q}) \in L^2(\widehat{\Lambda})$.

\medskip

Finally, it is straightforward to deduce an {\em average sampling formula} valid for any $T=\sum_{\lambda \in \Lambda} c(\lambda)\, \alpha_\lambda(S)\in V_{S}^2$. Indeed, 
for the sequence $\mathbf{c}=\{c(\lambda)\}_{\lambda \in \Lambda}\in \ell^2(\Lambda)$ we have the Riesz basis expansion
\[
\mathbf{c}=\sum_{\lambda \in \Lambda} \big\langle \mathbf{c}, T_\lambda \mathbf{q}^*\big\rangle_{\ell^2(\Lambda)}\, T_\lambda \mathbf{p}=\sum_{\lambda \in \Lambda} s_T(\lambda)\,T_\lambda \mathbf{p}\quad \text{in $\ell^2(\Lambda)$}\,.
\]
Applying the isomorphism $\mathcal{T}_S$ one gets that there exists a unique $H\in V_{S}^2$  such that for each $T \in V_{S}^2$ the average sampling formula
\begin{equation}
\label{fsamp1}
T=\sum_{\lambda \in \Lambda}s_T(\lambda)\, \alpha_\lambda(H) \quad \text{in $\mathcal{H}\mathcal{S}$-norm}
\end{equation}
holds. To be more precise, $H=L_h \in V_{S}^2$ with Weyl symbol $h=\mathcal{T}_{a_S}\mathbf{p} \in V_{a_S}^2$. Notice that 
$\mathcal{T}_{a_S}(T_\lambda \mathbf{p})=T_\lambda(\mathcal{T}_{a_S}\mathbf{p})$, where the same symbol $T_\lambda$ denotes the translation by 
$\lambda$ in $\ell^2(\Lambda)$ and in $L^2(\mathbb{R}^d\times \widehat{\mathbb{R}}^d)$ respectively. Furthermore, the convergence of the series in Hilbert-Schmidt norm is unconditional since 
$\{\alpha_\lambda (H)\}_{\lambda \in \Lambda}$ is a Riesz basis for $V_{S}^2$.

\medskip

The above result can be generalized and summarized as follows:
\begin{dfn}
\label{def1}
A generalized stable sampling procedure in $V_{S}^2$ is a map $\mathcal{S}_{\text{samp}}:V_{S}^2 \rightarrow \ell^2(\Lambda)$ defined as 
\[
T=\sum_{\lambda \in \Lambda} c(\lambda)\, \alpha_\lambda(S)\in V_{S}^2 \longmapsto \{s_T(\lambda)\}_{\lambda \in \Lambda} \text{ such that $s_T:= \mathbf{c} \ast_\Lambda \mathbf{q} \in \ell^2(\Lambda)$}\,,
\]
where $\mathbf{q} \in \ell^2(\Lambda)$ satisfies the conditions
\begin{equation}
\label{rieszcondition}
0<\einf_{\xi \in \widehat{\Lambda}} |\mathcal{F}_\sigma^\Lambda(\mathbf{q})(\xi)| \le \esup_{\xi \in \widehat{\Lambda}} |\mathcal{F}_\sigma^\Lambda(\mathbf{q})(\xi)| <\infty\,.
\end{equation}
\end{dfn}
Associated to a generalized stable sampling procedure $\mathcal{S}_{\text{samp}}$ in $V_{S}^2$ we obtain the following sampling result:
\begin{teo}
\label{tsamp1}
Assume that a generalized stable sampling procedure $\mathcal{S}_{\text{samp}}$ in $V_{S}^2$ is given as in Definition \ref{def1} with associated sequence $\mathbf{q} \in \ell^2(\Lambda)$. Then, there exists a unique Hilbert-Schmidt operator $H\in V_{S}^2$ such that the sampling formula $T=\sum_{\lambda \in \Lambda}s_T(\lambda)\, \alpha_\lambda(H)$  holds in $V_{S}^2$. The convergence of the series is unconditional in Hilbert-Schmidt norm.

Reciprocally, if a sampling formula like \eqref{fsamp1} holds in $V_{S}^2$ where $s_T(\lambda)= \big(\mathbf{c} \ast_\Lambda \mathbf{q} \big)(\lambda)$, 
$\lambda \in \Lambda$, and $\{\alpha_\lambda (H)\}_{\lambda \in \Lambda}$ is a Riesz basis for $V_{S}^2$, then the conditions in Eq.\eqref{rieszcondition} are satisfied.
\end{teo}
\begin{proof}
The first part of the proof has been done above. For the second part, observe that $\{\mathcal{T}_S^{-1}\big(\alpha_\lambda(H)\big)\}_{\lambda \in \Lambda}$ is a Riesz basis for $\ell^2(\Lambda)$ with dual basis $\{T_\lambda \mathbf{q}^*\}_{\lambda \in \Lambda}$ which implies conditions in \eqref{rieszcondition}.
\end{proof}
%%%%%%%%%%%%%%%%%%
\subsubsection*{Some comments} 
%%%%%%%%%%%%%%%%%%
Closing this section some comments are pertinent; the involved details can be found in Refs.~\cite{luef:18,skret:20}:
\begin{itemize}
\item The necessary and sufficient condition on $\{\alpha_\lambda (S)\}_{\lambda \in \Lambda}$ to be a Riesz sequence for $\mathcal{H}\mathcal{S}(\mathbb{R}^d)$ can be expressed also in terms of the symplectic Fourier series of the convolution of two operators. It reads:  The function 
$\mathcal{F}_\sigma^\Lambda(S\ast_\Lambda\check{S}^*)=P_{\Lambda^\circ}(|\mathcal{F}_W(S)|^2$ has no zeros in $\widehat{\Lambda}$. 

The {\em convolution of two operators $S, T$} is formally defined as the function 
\[
S\ast T(z):= {\rm tr} [S\alpha_z(\check{T})]\,, \quad z\in \mathbb{R}^{2d}\,, 
\]
where $\check{T}=PTP$ and $P$ denotes the {\em parity operator} $(P\phi)(t)=\phi(-t)$ for $\phi \in L^2(\mathbb{R}^d)$. Replacing $\mathbb{R}^{2d}$ by a lattice $\Lambda\subset \mathbb{R}^{2d}$ we obtain the convolution of two operators $S, T$ at $\Lambda$ as the sequence $S\ast_\Lambda T(\lambda):=S\ast T(\lambda)$ for $\lambda \in \Lambda$.

\item The {\em convolution of a function $f$ and an operator $S$} is formally defined by the operator-valued integral $f\ast S=\int_{\mathbb{R}^{2d}}f(z)\,\alpha_z(S)dz$. Replacing $\mathbb{R}^{2d}$ by a lattice $\Lambda\subset \mathbb{R}^{2d}$ we get the definition $\mathbf{c}\ast_\Lambda S:=S\ast_\Lambda \mathbf{c}:=\sum_{\lambda\in \Lambda} c(\lambda)\,\alpha_\lambda(S)$. As a consequence, the subspace $V_S^2$ can be also expressed  as 
\[
V_S^2=\ell^2(\Lambda) \ast_{\Lambda}S\,.
\]
 
\item The average sample $s_{_T}(\lambda):=\big\langle T, \alpha_\lambda(Q)\big \rangle_{_{\mathcal{H}\mathcal{S}}}$, $\lambda \in \Lambda$, can be expressed, under appropriate hypotheses (see, for instance, Ref.~\cite{skret:20}), as a convolution like in the classical shift-invariant  case. Indeed, for each $f\in L^2(\mathbb{R}^d)$ its average samples are 
\[
\big\langle f, \psi(\cdot-n)\big\rangle_{_{L^2(\mathbb{R})}}=f\ast \widetilde{\psi}(n)\,, \quad n\in \mathbb{Z}\,,
\]
where $\widetilde{\psi}(t)=\overline{\psi(-t)}$ is the average function. In the case treated here, an easy calculation gives 
\[
\big\langle T, \alpha_\lambda(Q)\big \rangle_{_{\mathcal{H}\mathcal{S}}}= {\rm tr}[T\alpha_\lambda(Q)^*]={\rm tr}[T\alpha_\lambda(Q^*)]=T\ast_\Lambda\check{Q}^*(\lambda)=T \ast_\Lambda \widetilde{Q}(\lambda)\,,\quad \lambda \in \Lambda\,,
\]
where $\widetilde{Q}=\check{Q}^*$. By using the convolution notations, the sampling formula \eqref{fsamp1} can be expressed by
\[
T=(T\ast_\Lambda \widetilde{Q}) \ast_\Lambda H\quad \text{for each $T\in V_S^2$}\,.
\]

\item It is worth to mention that the sampling results obtained in this paper can be also derived by using the Kohn-Niremberg symbol of a Hilbert-Schmidt operator instead of the Weyl symbol. If $a_{_S}$ denotes the Weyl symbol  of $S$, its Kohn-Niremberg $\sigma_{_S}$ is given by $Ua_{_S}$ where $U$ is the unitary operator in $L^2(\mathbb{R}^{2d})$ such that $\widehat{Ua_{_S}}(\xi, u)={\rm e}^{\pi i u\cdot\xi}\, \widehat{a}_{_S}(\xi, u)$, $(\xi, u)\in \mathbb{R}^{2d}$ (see the details in Ref.~\cite{grochenig:01}).

\end{itemize}

%%%%%%%%%%%%%%%%%%%%%%%%%%%%%%%%%%%%%%%%
\section{The case of multiple generators}
\label{section4}
%%%%%%%%%%%%%%%%%%%%%%%%%%%%%%%%%%%%%%%%
Given a fixed set ${\bf S}=\{S_1, S_2, \dots, S_N\} \subset \mathcal{H}\mathcal{S}(\mathbb{R}^d)$, assume now that $\{\alpha_\lambda (S_n)\}_{\lambda \in \Lambda;\, n=1,2,\dots,N}$ is a Riesz sequence for $\mathcal{H}\mathcal{S}(\mathbb{R}^d)$ where $\Lambda \subset \mathbb{R}^d\times \widehat{\mathbb{R}}^d$ is a full rank lattice with dual group $\widehat{\Lambda}$. Thus, we consider the closed subspace $V_{\bf S}^2$  in $\mathcal{H}\mathcal{S}(\mathbb{R}^d)$ given by
\[
V_{\bf S}^2=\Big\{\sum_{n=1}^N\sum_{\lambda \in \Lambda} c_n(\lambda)\, \alpha_\lambda(S_n)\,\, :\,\, \{c_n(\lambda)\}_{\lambda \in \Lambda}\in \ell^2(\Lambda)\,, n=1, 2, \dots, N \Big\}\,.
\]
For each $T=\sum_{n=1}^N\sum_{\lambda \in \Lambda} c_n(\lambda)\, \alpha_\lambda(S_n)$ in $V_{\bf S}^2$ we define a set of {\em generalized samples}
\[
\mathbf{s}_{_T}(\lambda)=\big(s_{_{T,1}}(\lambda), s_{_{T,2}}(\lambda), \dots, s_{_{T,M}}(\lambda)\big)^\top\,,\quad \lambda \in \Lambda\,,
\] 
by means of a {\em discrete convolution system} associated to a matrix $A=[a_{m,n}] \in \mathcal{M}_{_{M\times N}}\big(\ell^2(\Lambda)\big)$, i.e., an $M\times N$ matrix with entries in $\ell^2(\Lambda)$, as follows
\[
T=\sum_{n=1}^N\sum_{\lambda \in \Lambda} c_n(\lambda)\, \alpha_\lambda(S_n) \in V_{\bf S}^2 \longmapsto \mathbf{s}_{_T}(\lambda):=\big(A\ast_\Lambda \mathbf{c}\big)(\lambda)=\sum_{\lambda'\in \Lambda} A(\lambda-\lambda')\, \mathbf{c}(\lambda'),\quad \lambda \in \Lambda\,,
\]
where $\mathbf{c}=(c_1, c_2, \dots, c_N)^\top \in \ell^2_{_N}(\Lambda):=\ell^2(\Lambda)\times \dots \times \ell^2(\Lambda)$ ($N$ times). Note that the $m$-th entry of \,$A \ast_\Lambda \mathbf{c}$\, is\, $s_{_{T,m}}=\sum_{n=1}^N (a_{m,n}\ast_\Lambda c_{n})$. 

\medskip

As in Section \ref{section3}, it is easy to deduce that these samples generalize the average samples
\[
\mathbf{s}_{_T}(\lambda)=\big(\langle T, \alpha_\lambda(Q_1) \rangle_{_{\mathcal{H}\mathcal{S}}}, \langle T, \alpha_\lambda(Q_2) \rangle_{_{\mathcal{H}\mathcal{S}}}, \dots, \langle T, \alpha_\lambda(Q_M) \rangle_{_{\mathcal{H}\mathcal{S}}}\big)^\top\,,\quad \lambda \in \Lambda\,,
\]
obtained from $M$ fixed operators $Q_1, Q_2, \dots, Q_M$ in $\mathcal{H}\mathcal{S}(\mathbb{R}^d)$, not necessarily in $V_{\bf S}^2$. Indeed, for  the $m$-th component of the sample $\mathbf{s}_{_T}(\lambda)$, $\lambda \in \Lambda$, we have
\[
\begin{split}
s_{_{T,m}}(\lambda):=&\big\langle T, \alpha_\lambda(Q_m) \big\rangle_{_{\mathcal{H}\mathcal{S}}}=\big\langle a_{_T}, T_\lambda a_{_{Q_m}} \big\rangle_{_{L^2(\mathbb{R}^d\times \widehat{\mathbb{R}}^d)}}=
\big\langle \sum_{n=1}^N\sum_{\lambda' \in \Lambda} c_n(\lambda')\, T_{\lambda'}  a_{_{S_n}}, T_\lambda a_{_{Q_m}} \big\rangle_{_{L^2(\mathbb{R}^d\times \widehat{\mathbb{R}}^d)}} \\
=&\sum_{n=1}^N\sum_{\lambda' \in \Lambda} c_n(\lambda') \langle T_{\lambda'}  a_{_{S_n}}, T_\lambda a_{_{Q_m}} \rangle_{_{L^2(\mathbb{R}^d\times \widehat{\mathbb{R}}^d)}} =\sum_{n=1}^N\sum_{\lambda' \in \Lambda} c_n(\lambda') \langle a_{_{S_n}}, T_{\lambda-\lambda'} a_{_{Q_m}} \rangle_{_{L^2(\mathbb{R}^d\times \widehat{\mathbb{R}}^d)}} \\
=&\sum_{n=1}^N \big(a_{m,n} \ast_\Lambda c_n\big)(\lambda)\,,
\end{split}
\]
where $a_{m,n}(\lambda):=\big\langle a_{_{S_n}}, T_\lambda a_{_{Q_m}} \big\rangle_{_{L^2(\mathbb{R}^{2d})}}$, \,$\lambda \in \Lambda$, and $a_{_{S_n}}$, $a_{_{Q_m}}$ are the Weyl symbols of $S_n$, $Q_m$ respectively. 
             \medskip

The main needed properties of a {\em discrete convolution system} $\mathcal{A}$ with associated matrix $A=[a_{m,n}] \in \mathcal{M}_{_{M\times N}}\big(\ell^2(\Lambda)\big)$ given by
\begin{equation}
\label{syscon}
\begin{array}{rccl}
	&\mathcal{A}: \ell^2_{_N}(\Lambda)  &\longrightarrow &\ell^2_{_M}(\Lambda)\\
	  & \mathbf{c} &\longmapsto &  \mathcal{A}(\mathbf{c})=A\ast_\Lambda \mathbf{c}\,,
    \end{array}
\end{equation} 
are summarized below. The details and proofs can be found, for instance, in Refs.~\cite{garcia:19,gerardo:19}.

\begin{enumerate}
\item $\mathcal{A}$ is a well-defined bounded operator if and only if the matrix $\widehat{A}\in \mathcal{M}_{_{M\times N}}\big(L^\infty(\widehat{\Lambda})\big)$, where $\widehat{A}(\xi):=\big[\mathcal{F}_\sigma^\Lambda(a_{m,n})(\xi)\big]$, a.e. $\xi \in \widehat{\Lambda}$, denotes the {\em transfer matrix} of $A$ (along this section we identify the matrix $\widehat{A}$ with entries in $L^\infty(\widehat{\Lambda})$ and the essentially bounded matrix-valued function 
$\widehat{A}(\xi)$, a.e. $\xi \in \widehat{\Lambda}$\,). Having in mind the equivalence between the spectral and Frobenius norms for matrices (see Ref.~\cite{horn:99}), the above condition is equivalent to the new condition 
\[
\beta_A:= \esup_{\xi\in \widehat{\Lambda}} \lambda_{\text{max}}[\widehat{A}(\xi)^*\widehat{A}(\xi)]<+\infty\,, 
\]
where 
$\lambda_{\text{max}}$ denotes the largest eigenvalue of the positive semidefinite matrix $\widehat{A}(\xi)^*\widehat{A}(\xi)$.

\item Its adjoint operator $\mathcal{A}^*: \ell^2_{_M}(\Lambda) \rightarrow \ell^2_{_N}(\Lambda)$ is also a bounded convolution system with associated matrix $A^*=\big[a^*_{m,n}\big]^\top \in  \mathcal{M}_{_{N\times M}}\big(\ell^2(\Lambda)\big)$, where $a^*_{m,n}$ denotes the involution $a^*_{m,n}(\lambda):=\overline{a_{m,n}(-\lambda)}$, $\lambda\in \Lambda$. Its transfer matrix is $\widehat{A^*}(\xi)$ is just the transpose conjugate of $\widehat{A}(\xi)$, i.e., $\widehat{A}(\xi)^*$, a.e. $\xi \in \widehat{\Lambda}$.

\item The bounded operator $\mathcal{A}$ is injective with a closed range if and only if the operator $\mathcal{A}^*\,\mathcal{A}$ is invertible; equivalently, if and only if the constant 
\[
\alpha_A:=\einf_{\xi\in \widehat{\Lambda}} \lambda_{\text{min}}[\,\widehat{A}(\xi)^*\widehat{A}(\xi)]>0\,,
\]
where $\lambda_{\text{min}}$ denotes the smallest eigenvalue of the positive semidefinite matrix $\widehat{A}(\xi)^*\widehat{A}(\xi)$. Equivalently, we have
$\delta_A:=\einf_{\xi\in \widehat{\Lambda}} \det[\,\widehat{A}(\xi)^*\widehat{A}(\xi)]>0$.

\item The bounded operator $\mathcal{A}$ is an isomorphism if and only if $M=N$ and the constant $\einf_{\xi \in \widehat{\Lambda}}\big|\det [\widehat{A}(\xi)]\big|>0$.

\end{enumerate}

\noindent Besides, discrete convolution systems are intimately related with translations $T_\lambda$ in $\ell^2_{_N}(\Lambda)$; remind that for $\mathbf{c}\in \ell^2_{_N}(\Lambda)$,  $T_{\lambda}\,\mathbf{c}(\lambda')=\mathbf{c}(\lambda'-\lambda)$, $\lambda' \in \Lambda$. Indeed, let $\mathbf{a}^*_{m}$ denote the $m$-th column of the matrix $A^*$, then the $m$-th component of  $\mathcal{A}(\mathbf{c})$ is
\[
[A\ast \mathbf{c}]_{m}(\lambda)=\sum_{n=1}^N (a_{m,n} \ast_{_\Lambda} c_n)(\lambda) =\big\langle \mathbf{c}, T_{\lambda}\,\mathbf{a}^*_{m} \big\rangle_{\ell^2_{_N}(\Lambda)}\,, \quad \lambda \in \Lambda\,.
\]
As a consequence:

\begin{enumerate}[(i)]

\item  The operator $\mathcal{A}$ is the {\em analysis operator} of the sequence $\big\{T_{\lambda} \mathbf{a}^*_{m}\big\}_{\lambda \in \Lambda;\, m=1,2,\dots,M}$ in $\ell^2_{_N}(\Lambda)$. Thus, the sequence $\big\{T_\lambda\,\mathbf{a}_m^*\big\}_{\lambda \in \Lambda;\,m=1, 2,\dots,M}$ is a Bessel sequence in $\ell_{_N}^2(\Lambda)$ if and only if the convolution system $\mathcal{A}$ is bounded, or equivalently, if and only $\beta_A<+\infty$.

\item The sequence $\big\{T_{\lambda}\,\mathbf{a}^*_{m}\big\}_{\lambda \in \Lambda;\, m=1,2,\dots,M}$  is a frame for $\ell^2_{_N}(\Lambda)$ if and only if its bounded analysis operator is injective with a closed range (see Ref.~\cite{ole:16}).  Therefore, it will be a frame for $\ell^2_{N}(\Lambda)$ if and only if 
\[
0<\alpha_A:=\einf_{\xi\in \widehat{\Lambda}} \lambda_{\text{min}}[\,\widehat{A}(\xi)^*\widehat{A}(\xi)] \le \beta_A:= \esup_{\xi\in \widehat{\Lambda}} \lambda_{\text{max}}[\widehat{A}(\xi)^*\widehat{A}(\xi)]<+\infty\,.
\]

\item Concerning the duals of $\big\{T_{\lambda}\,\mathbf{a}^*_{m}\big\}_{\lambda \in \Lambda;\, m=1,2,\dots,M}$ having the same structure, consider two matrices $\widehat{A}\in \mathcal{M}_{_{M\times N}}(L^\infty(\widehat{\Lambda}))$ and $\widehat{B}\in \mathcal{M}_{_{N\times M}}(L^\infty(\widehat{\Lambda}))$, and let $\mathbf{b}_{m}$ denotes the $m$-th column of the matrix $B$ associated to $\widehat{B}$. Then, the sequences $\big\{T_\lambda\,\mathbf{a}^*_{m}\big\}_{\lambda \in \Lambda;\, m=1,2,\dots,M}$ and $\big\{T_\lambda\,\mathbf{b}_{m}\big\}_{\lambda \in \Lambda;\, m=1,2,\dots,M}$ form a pair of dual frames for $\ell^2_{_N}(\Lambda)$ if and only if $\widehat{B}(\xi)\,\widehat{A}(\xi)=I_{_N}$,\,\, a.e. $\xi \in \widehat{\Lambda}$; equivalently, if and only if $\mathcal{B}\,\mathcal{A}=\mathcal{I}_{\ell^2_{N}(\Lambda)}$, i.e., the convolution system 
$\mathcal{B}$ is a left-inverse of the convolution system $\mathcal{A}$ (see Ref.~\cite{garcia:19}). Thus in $\ell^2_{_N}(\Lambda)$ we have the frame expansion
\[
\mathbf{c}=\sum_{m=1}^M \sum_{\lambda\in \Lambda} \big\langle \mathbf{c}, T_\lambda \mathbf{a}^*_{m} \big \rangle_{\ell^2_{_N}(\Lambda)}\, T_\lambda \mathbf{b}_{m}\quad \text{for each $\mathbf{c} \in \ell^2_{_N}(\Lambda)$}\,.
\]
Note that a possible left-inverse of the matrix $\widehat{A}(\xi)$ is given by its Moore-Penrose pseudo-inverse $\widehat{A}(\xi)^\dag=\big[\widehat{A}(\xi)^*\widehat{A}(\xi)\big]^{-1}\widehat{A}(\xi)^*$, a.e. $\xi \in \widehat{\Lambda}$.

\item For the case $M=N$ the sequence  $\big\{T_\lambda\,\mathbf{a}_m^*\big\}_{\lambda\in \Lambda;\,m=1, 2,\dots,N}$ is a Riesz basis for $\ell_{_N}^2(\Lambda)$. The square matrix $\widehat{A}(\xi)$ is invertible, a.e. $\xi \in \widehat{\Lambda}$, and from the columns of $\widehat{A}(\xi)^{-1}$ we get its dual Riesz basis $\big\{T_\lambda\,\mathbf{b}_m\big\}_{\lambda\in \Lambda;\,m=1, 2,\dots,N}$.

\end{enumerate}

\medskip

Now suppose that a sampling procedure is given in $ V_{\bf S}^2$ by means of a discrete convolution system $\mathcal{A}$, i.e., 
\[
T=\sum_{n=1}^N\sum_{\lambda \in \Lambda} c_n(\lambda)\, \alpha_\lambda(S_n) \in V_{\bf S}^2 \longmapsto \mathbf{s}_{_T}:=A\ast_\Lambda \mathbf{c} \in \ell^2_{_M}(\Lambda)\,,
\]
and assume, in the light of the above discussion, that the sequence $\big\{T_\lambda\,\mathbf{a}^*_{m}\big\}_{\lambda \in \Lambda;\, m=1,2,\dots,M}$ is a frame for $\ell^2_{_N}(\Lambda)$ with a dual frame $\big\{T_\lambda\,\mathbf{b}_{m}\big\}_{\lambda \in \Lambda;\, m=1,2,\dots,M}$. Then, we can recover any $T=\sum_{n=1}^N\sum_{\lambda \in \Lambda} c_n(\lambda)\, \alpha_\lambda(S_n) \in V_{\bf S}^2$ from its samples $\{\mathbf{s}_{_T}(\lambda)\}_{\lambda \in \Lambda}$ by means of a frame expansion. Indeed, for the coefficients $\mathbf{c}=(c_1, c_2, \dots, c_{_N})^\top \in \ell^2_{_N}(\Lambda)$ of $T$  we have
\begin{equation}
\label{cexpan}
\mathbf{c}=\sum_{m=1}^M \sum_{\lambda\in \Lambda} \big\langle \mathbf{c}, T_\lambda \mathbf{a}^*_{m} \big \rangle_{\ell^2_{_N}(\Lambda)}\, T_\lambda \mathbf{b}_{m}=\sum_{m=1}^M \sum_{\lambda\in \Lambda} s_{_{T,m}}(\lambda)\,T_\lambda \mathbf{b}_{m}
\quad \text{in $\ell^2_{_N}(\Lambda)$}\,.
\end{equation}
Consider the corresponding isomorphism $\mathcal{T}_{\bf S}$ in \eqref{isos} which in this case reads:
\[
\begin{array}[c]{cccccc}
 \mathcal{T}_{\bf S}:&  \ell^2_{_N}(\Lambda) & \longrightarrow & V_{a_{\bf S}}^2 \subset L^2(\mathbb{R}^{2d}) & \longrightarrow & V_{\bf S}^2 \subset \mathcal{H}\mathcal{S}(\mathbb{R}^d)
 \phantom{\dfrac{a}{b}} \\
       & \mathbf{c}=(c_1, c_2, \dots, c_N)^\top & \longmapsto &\displaystyle{\sum_{n=1}^N \sum_{\lambda \in \Lambda} c_n(\lambda)\, T_\lambda  a_{_{S_n}}}  & \longmapsto  &\displaystyle{\sum_{n=1}^N \sum_{\lambda \in \Lambda} c_n(\lambda)\, \alpha_\lambda(S_n)}\,,
\end{array}
\]
where $a_{_{S_n}}$ denotes the Weyl symbol of $S_n$, $n=1,2, \dots,N$.
Applying the isomorphism $\mathcal{T}_{\bf S}$ in expansion \eqref{cexpan}, for each $T\in V_{\bf S}^2$ we obtain the sampling expansion
\[
T=\sum_{m=1}^M \sum_{\lambda\in \Lambda} s_{_{T,m}}(\lambda)\, \alpha_\lambda(H_m)\quad \text{in $\mathcal{H}\mathcal{S}$-norm}\,,
\] 
where $H_m=L_{h_m} \in V_{\bf S}^2$ with Weyl symbol $h_m=\mathcal{T}_{a_{\bf S}}(\mathbf{b}_m) \in V_{a_{\bf S}}^2$, $m=1,2,\dots,M$. Furthermore, the convergence of the series in the Hilbert-Schmidt norm is unconditional since 
$\{\alpha_\lambda (H_m)\}_{\lambda \in \Lambda;\, m=1,2,\dots,M}$ is a frame for $V_{\bf S}^2$.

\medskip

The above result can be summarized as follows:
\begin{dfn}
\label{def2}
A generalized stable sampling procedure in $V_{\bf S}^2$ is a map $\mathcal{S}_{\text{samp}}:V_{\bf S}^2 \rightarrow \ell^2_{_M}(\Lambda)$ defined as
\[
T=\sum_{m=1}^M\sum_{\lambda \in \Lambda} c(\lambda)\, \alpha_\lambda(S)\in V_{\bf S}^2 \longmapsto {\bf s_T}:=A \ast_\Lambda \mathbf{c} \in  \ell^2_{_M}(\Lambda)\,,
\]
where the matrix $A=[a_{m,n}] \in \mathcal{M}_{_{M\times N}}\big(\ell^2(\Lambda)\big)$ satisfies the conditions:
\begin{equation}
\label{framecondition}
0<\alpha_A:=\einf_{\xi\in \widehat{\Lambda}} \lambda_{\text{min}}[\,\widehat{A}(\xi)^*\widehat{A}(\xi)] \le \beta_A:= \esup_{\xi\in \widehat{\Lambda}} \lambda_{\text{max}}[\widehat{A}(\xi)^*\widehat{A}(\xi)]<+\infty\,.
\end{equation}
\end{dfn}
Associated with a generalized stable sampling procedure $\mathcal{S}_{\text{samp}}$ in $V_{\bf S}^2$  we obtain the following sampling result:
\begin{teo}
\label{tsamp2}
Assume that a generalized stable sampling procedure $\mathcal{S}_{\text{samp}}$  with associated matrix $A$  is given in $V_{\bf S}^2$ as in Definition \ref{def2}. Then, there exist $M\geq N$ elements $H_m\in V_{\bf S}^2$, $m=1,2,\dots,M$, such that the sampling formula 
\begin{equation}
\label{fsamp2}
T=\sum_{m=1}^M\sum_{\lambda \in \Lambda} s_{_{T,m}}(\lambda)\, \alpha_\lambda(H_m) \quad \text{in $\mathcal{H}\mathcal{S}$-norm}
\end{equation}
holds for each $T\in V_{\bf S}^2$ where $\{\alpha_\lambda (H_m)\}_{\lambda \in \Lambda;\, m=1,2,\dots,M}$ is a frame for $V_{\bf S}^2$. The convergence of the series is unconditional in Hilbert-Schmidt norm.

Reciprocally, if a sampling formula like \eqref{fsamp2} holds in $V_{\bf S}^2$ where 
\[
{\bf s_T}(\lambda)=\big(s_{_{T,1}}(\lambda), s_{_{T,2}}(\lambda), \dots, s_{_{T,M}}(\lambda)\big)^\top:=\big(A\ast_\Lambda \mathbf{c}\big)(\lambda)\,, \quad
\lambda \in \Lambda\,, 
\]
with $\beta_A<+\infty$, and $\{\alpha_\lambda (H_m)\}_{\lambda \in \Lambda;\, m=1,2,\dots,M}$ is a frame for $V_{\bf S}^2$, then the left-hand condition in \eqref{framecondition} also holds.
\end{teo}
\begin{proof}
The first part of the theorem has been proved above. Observe that the operators $H_m\in V_{\bf S}^2$, $m=1,2,\dots,M$, depend on the dual frames 
$\big\{T_\lambda\,\mathbf{b}_{m}\big\}_{\lambda \in \Lambda;\, m=1,2,\dots,M}$ of the frame $\big\{T_{\lambda}\,\mathbf{a}^*_{m}\big\}_{\lambda \in \Lambda;\, m=1,2,\dots,M}$. Namely, $H_m=L_{h_m} \in V_{\bf S}^2$ with Weyl symbol $h_m=\mathcal{T}_{a_{\bf S}}(\mathbf{b}_m) \in V_{a_{\bf S}}^2$, $m=1,2,\dots,M$; 
$\mathbf{b}_m$ denotes the $m$-th  column of the matrix $B$ with transfer matrix $\widehat{B}$.
There are infinite dual frames  whenever $M>N$; they are obtained  from the left-inverses $\widehat{B}(\xi)$ of $\widehat{A}(\xi)$, i.e.,
$\widehat{B}(\xi)\,\widehat{A}(\xi)=I_{_N}$,\,\, a.e. $\xi \in \widehat{\Lambda}$, which are obtained, from the Moore-Penrose pseudo-inverse $\widehat{A}(\xi)^\dag$, by means of the $N\times M$ matrices 
\[
\widehat{B}(\xi):=\widehat{A}(\xi)^\dag+C(\xi)\big[I_M-\widehat{A}(\xi)\widehat{A}(\xi)^\dag\big]\,, \text{\, \,a.e. $\xi \in \widehat{\Lambda}$}\,, 
\]
where $C$ denotes any $N\times M$ matrix with entries in $L^\infty(\widehat{\Lambda})$.

For the second part, we have that $\{\mathcal{T}_{\bf S}^{-1}[\alpha_\lambda(H_m)]\}_{\lambda \in \Lambda;\, m=1,2,\dots,M}$ is a frame for $\ell^2_{_N}(\Lambda)$ and
\[
\mathbf{c}=\sum_{m=1}^M\sum_{\lambda \in \Lambda} s_{_{T,m}}(\lambda)\, \mathcal{T}_{\bf S}^{-1}[\alpha_\lambda(H_m)]=\sum_{m=1}^M\sum_{\lambda \in \Lambda} \,\big\langle \mathbf{c}, T_\lambda \mathbf{a}^*_{m} \big \rangle_{\ell^2_{_N}(\Lambda)} \mathcal{T}_{\bf S}^{-1}[\alpha_\lambda(H_m)]\quad \text{in $\ell^2_{_N}(\Lambda)$}\,.
\]
Since $\beta_A<+\infty$, the sequence $\big\{T_{\lambda}\,\mathbf{a}^*_{m}\big\}_{\lambda \in \Lambda;\, m=1,2,\dots,M}$ is a Bessel sequence for $\ell^2_{_N}(\Lambda)$, and consequently (see Ref.~\cite[Lemma 6.3.2]{ole:16}), a dual frame of
$\{\mathcal{T}_{\bf S}^{-1}[\alpha_\lambda(H_m)]\}_{\lambda \in \Lambda;\, m=1,2,\dots,M}$; hence, $\alpha_A>0$.
\end{proof}
Notice that in Theorem \ref{tsamp2} necessarily $M\geq N$; in case $M=N$ more can be said:
\begin{cor}
In case $M=N$, the following statements are equivalent:
\begin{enumerate}
\item 
\[
0<\einf_{\xi\in \widehat{\Lambda}} \big|\det [\widehat{A}(\xi)]\big| \le \esup_{\xi\in \widehat{\Lambda}} \big|\det [\widehat{A}(\xi)]\big|<+\infty
\] 
\item There exist $N$ unique elements $H_n$, $n=1,2,\dots, N$, in $V_{\bf S}^2$ such that the associated sequence
$\big\{\alpha_\lambda(H_{n})\big\}_{\lambda\in \Lambda;\,n=1,2,\dots, N}$ is a Riesz basis for $V_{\bf S}^2$ and the sampling formula
\[
T=\sum_{n=1}^N\sum_{\lambda \in \Lambda} s_{_{T,n}}(\lambda)\, \alpha_\lambda(H_n)\ \quad \text{ in $\mathcal{H}\mathcal{S}$}
\]
holds for each $T\in V_{\bf S}^2$.
\end{enumerate}
Moreover, the interpolation property $s_{_{H_n,n'}}(\lambda)=\delta_{n,n'}\delta_{\lambda,0}$, where $\lambda \in \Lambda $ and $n,n'=1,2,\dots,N$, holds.
\end{cor}
\begin{proof}
In this case, the square matrix $\widehat{A}(\xi)$ is invertible and statement $1.$  is equivalent to $0<\alpha_A \le \beta_A <+\infty$; besides,  
any Riesz basis has a unique dual basis. The uniqueness of the coefficients in a Riesz basis expansion gives the interpolation property.
\end{proof}

\bigskip

\noindent{\bf Acknowledgments:}
The author thanks {\em Universidad Carlos III de Madrid} for granting him a sabbatical year in 2020-21.
This work has been supported by the grant MTM2017-84098-P from the Spanish {\em Ministerio de Econom\'{\i}a y Competitividad (MINECO)}.

%%%%%%%%%%%%%%%%%%%%%%%%%%%%%%%%%%%%%%%%%%%%%%%%%%%%%%%%%%%%%%%%%%%%%

%%%%%%%%%%%%%%%%%%%%%%%%%
\end{document}